\documentclass{amsart}
\usepackage{amsfonts,amssymb,amscd,amsmath,enumerate,verbatim,calc}
\usepackage{color}
\usepackage[colorlinks]{hyperref}
\usepackage{tikz}
\usetikzlibrary{arrows}

\newtheorem{thm}{Theorem}[section]
\newtheorem{cor}[thm]{Corollary}
\newtheorem{lem}[thm]{Lemma}

\theoremstyle{definition}

\theoremstyle{remark}
\newtheorem{rem}[thm]{Remark}
\numberwithin{equation}{section}

\begin{document}

\author[H. Sahleh \lowercase{and} A. A. Alijani]{ H\lowercase{ossein} Sahleh  \lowercase{and} A\lowercase{li} A\lowercase{kbar} Alijani}
\title[E\lowercase{xtensions of locally compact abelian, torsion-free groups by compact torsion abelian groups }] {E\lowercase{xtensions of locally compact abelian, torsion-free groups by compact torsion abelian groups}}


\maketitle

\begin{center}
 {\it Department of Pure Mathematics, Faculty of Mathematical Sciences,\\ University of Guilan,
Rasht, Iran}

\email{sahlehg@gmail.com}
\end{center}

\begin{abstract}
Let $X$ be a compact torsion abelian group. In this paper, we show that an extension of $F_{p}$ by $X$ splits where $F_{p}$ is the p-adic number group and $p$ a prime number. Also, we show that an extension of a torsion-free, non-divisible LCA group by $X$ is not split.
\end{abstract}

\maketitle
\section*{Introduction}

 Throughout, all groups are Hausdorff  abelian topological groups and will be written additively. Let $\pounds$ denote the category of locally compact abelian (LCA) groups with continuous homomorphisms as morphisms. The Pontrjagin dual of a group $G$ is denoted by $\hat{G}$. A morphism is called proper if it is open onto its image and a short exact sequence $0\to A\stackrel{\phi}{\to} B\stackrel{\psi}{\to}C\to 0$ in $\pounds$ is said to be proper exact if $\phi$ and $\psi$ are proper morphisms. In this case the sequence is called an extension of $A$ by $C$ ( in $\pounds$). Following \cite{FG1}, we let $Ext(C,A)$ denote the (discrete) group of extensions of $A$ by $C$. The splitting problem in LCA groups is finding conditions on $A$ and $C$ under which $Ext(C,A)=0$. In \cite{F1,F2,FG2} the splitting problem is studied. We have studied the splitting problem in the category of divisible, LCA groups \cite{SA1}. By using the splitting problem, we determined the LCA groups $G$ such that the maximal torsion subgroup of $G$ is closed \cite{SA2}. Let $X$ be a compact torsion group. In \cite[Theorem 1]{F2}, it is proved that if $G$ is a divisible LCA group, then $Ext(X,G)=0$. However, the suggested proof in \cite{F2} appears to be incomplete as it uses the incorrect Proposition 8 of \cite{F1}. In \cite{SA1}, we proved that if $G$ is a divisible, $\sigma-$compact group, then $Ext(X,G)=0$. Let $P$ be the set of all prime numbers, $J_{p}$, the p-adic integer group and $F_{p}$, the p-adic number group which is the minimal divisible extension of $J_{p}$  for every $p\in P$ \cite{HR}. By \cite[25.33]{HR}, a divisible, torsion-free LCA group $G$ has the form $G\cong \Bbb R^{n}\bigoplus A\bigoplus M\bigoplus E$, where $A$ is a discrete, torsion-free, divisible group, $M$ a compact connected torsion-free group and $E$, the minimal divisible extension of $\prod_{p\in P}J_{p}^{n_{p}}$ where $n_{p}$ is a cardinal number for every $p\in P$. In general, $E\neq \prod_{p\in P}F_{p}^{n_{p}}$. Let $I$ be a finite subset of $P$ such that
 $$n_{p}=
\left\{
	\begin{array}{ll}
		0  & \mbox{if } p\not\in I \\
		1 & \mbox{if } p\in I
	\end{array}
\right.$$

Then the minimal divisible extension of $\prod_{p\in I}J_{p}^{n_{p}}$ is $\prod_{p\in I}F_{p}^{n_{p}}$. In this paper, we show that if $G\cong \Bbb R^{n}\bigoplus A\bigoplus M\bigoplus \prod_{p\in I} F_{p}^{n_{p}}$, then $Ext(X,G)=0$ (see Theorem \ref{7}).

 The additive topological group of real numbers is denoted by $\Bbb R$, $\Bbb Q$ is the group of rationales with discrete topology and $\Bbb Z$ is the group of integers. The topological isomorphism will be denote by " $\cong$ ". For more on locally compact abelian groups see \cite{HR}.

\newcommand{\stk}[1]{\stackrel{#1}{\longrightarrow}}
\section{Extensions of a torsion-free LCA group by a compact torsion abelian group}
\begin{lem}\label{1}
Let $X\in\pounds$ and $p$ a prime number. Then $nExt(X,F_{p})=Ext(X,F_{p})$ for every positive integer $n$.
\end{lem}
\begin{proof}
Let $n$ be a positive integer and $f:F_{p}\to F_{p}$, $f(x)=nx$ for all $x\in F_{p}$. By \cite[Lemma 2]{A}, $f$ is open. So $f$ is a proper morphism. Consider the exact sequence $0\to Ker f\to F_{p}\stk{f} F_{p}\to 0$. By Corollary 2.10 of \cite{FG1}, we have the exact sequence
\begin{eqnarray}\label{2}
\to Ext(X,Ker f)\to Ext (X,F_{p})\stk{f_{*}} Ext (X,F_{p})\to 0
\end{eqnarray}\\
Since $f_* (Ext(X,F_{p}))=n Ext(X,F_{p})$, it follows from sequence \eqref{2} that $nExt(X,F_{p})=Ext(X,F_{p})$.
\end{proof}

\begin{lem}\label{3}
Let $X$ be a compact torsion group. Then $Ext(X,F_{p})=0$.
\end{lem}
\begin{proof}
$F_{p}$ is a totally disconnected group. So, by Theorem 24.30 of \cite{HR}, $F_{p}$ contains a compact open subgroup $K$. Now we have the following exact sequence
\begin{eqnarray}\label{4}
...\to Ext(X,K)\to Ext(X,F_{p})\to Ext(X,F_{p}/K)\to 0
\end{eqnarray}
Since $F_{p}$ is divisible, so $Ext(X,F_{p}/K)=0$ (see \cite[Theorem 3.4]{FG1}). Since $X$ is compact and  torsion, so by \cite[Theorem 25.9]{HR}, $nX=0$ for some positive integer $n$. Hence, $nExt(X,K)=0$ (see \cite[Lemma 2.5]{L}). Since \eqref{4} is exact, so $nExt(X,F_{p})=0$. Hence by Lemma \ref{1}, $Ext(X,F_{p})=0$.
\end{proof}
\begin{rem}\label{5}
Let $X$ be a group and $f:X\to X$,$f(x)=nx$ for all $x\in X$. If $f$ is a topological isomorphism for every positive integer $n$, then $X$ is a divisible, torsion-free group.
\end{rem}
\begin{thm}\label{6}
Let $X$ be a compact group and $p$ a prime number. Then $Ext(X,F_{p})$ is a divisible, torsion-free group.
\end{thm}
\begin{proof}
Let $n$ be a positive integer. Then the exact sequence $0\to X\stackrel{\times n}{\to} X\stackrel{}{\to} X/nX\to 0$ induces the following exact sequence $$Ext(X/nX,F_{p})\to Ext(X,F_{p})\stackrel{\times n}\to Ext(X,F_{p})\to 0$$ By Lemma \ref{3}, $Ext(X/nX,F_{p})=0$. So $Ext(X,F_{p})\stackrel{\times n}\to Ext(X,F_{p})$ is a topological isomorphism. Hence by Remark \ref{5}, $Ext(X,F_{p})$ is a divisible, torsion-free group.
\end{proof}
\begin{cor}
Let $X\in \pounds$. Then $Ext(X,F_{p})$ is a divisible, torsion-free group.
\end{cor}
\begin{proof}
Let $X\in \pounds$. By \cite[Theorem 24.30]{HR}, $X=\Bbb R^{n}\bigoplus H$ where $H$ contains a compact open subgroup $K$. Consider the exact sequence $$Ext(H/K,F_{p})\to Ext(H,F_{p})\to Ext(K,F_{p})\to 0$$ Since $H/K$ is a discrete group and $F_{p}$ a divisible group, so $Ext(H/K,F_{p})=0$. Hence $Ext(H,F_{p})\cong Ext(K,F_{p})$. By Theorem \ref{6}, $Ext(K,F_{p})$ is a divisible, torsion-free group. So $Ext(X,F_{p})$ is a divisible, torsion-free group.
\end{proof}
\begin{thm}\label{7}
Let $X$ be a compact torsion group and $G\cong \Bbb R^{n}\bigoplus A\bigoplus M\bigoplus \prod_{p\in I} F_{p}^{n_{p}}$ where $I$ is a finite subset of $P$ defined as follows:
 $$n_{p}=
\left\{
	\begin{array}{ll}
		0  & \mbox{if } p\not\in I \\
		1 & \mbox{if } p\in I
	\end{array}
\right.$$
Then $Ext(X,G)=0$.
\end{thm}
\begin{proof}
First recall that by \cite[Theorem 2.13]{FG1}, $$Ext(X,G)\cong Ext(X,A)\bigoplus Ext(X,M)\bigoplus \prod_{p\in I}Ext(X,F_{p})$$ Since $X$ is a totally disconnected group, so by \cite[Theorem 3.4]{FG1}, $Ext(X,A)=0$. Also $Ext(X,M)\cong Ext(\hat{M},\hat{X})$. Since $\hat{X}$ is a discrete bounded group and $\hat{M}$ a discrete torsion-free group, so by \cite[Theorem 27.5]{F},$Ext(\hat{M},\hat{X})=0$. By Lemma \ref{3}, $Ext(X,F_{p})=0$. Hence $Ext(X,G)=0$.
\end{proof}
\begin{lem}\label{8}
Let $X$ be a compact torsion group. Then $Hom(X,\Bbb Q/\Bbb Z)\cong \hat{X}$.\\
\end{lem}
\begin{proof}
The exact sequence $0\to \Bbb Z\to \Bbb Q\to \Bbb Q/\Bbb Z$ induces the following exact sequence $$Hom(X,\Bbb Q)\to Hom(X,\Bbb Q/\Bbb Z)\to Ext(X,\Bbb Z)\to Ext(X,\Bbb Q)$$ Since $X$ is torsion and $\Bbb Q$ is torsion-free, so $Hom(X,\Bbb Q)=0$. Also by \cite[Theorem 3.4]{FG1}, $Ext(X,\Bbb Q)=0$. Hence $Hom(X,\Bbb Q/\Bbb Z)\cong Ext(X,\Bbb Z)$. By \cite[Theorem 2.12 and Proposition 2.17]{FG1}, $Ext(X,\Bbb Z)\cong Ext(\hat{\Bbb Z},\hat{X})\cong \hat{X}$. So $Hom(X,\Bbb Q/\Bbb Z)\cong \hat{X}$.
\end{proof}
\begin{thm}\label{9}
Let $X$ be a compact torsion group and $G$ a torsion-free, non-divisible group. Then $Ext(X,G)\neq 0$.
\end{thm}
\begin{proof}
Let $G^{\ast}$ be the minimal divisible extension of $G$. By \cite[A.13]{HR}, $G^{\ast}$ is a divisible, torsion-free group. Since $X$ is torsion and $G^{\ast}$ torsion-free, so $Hom(X,G^{\ast})=0$. By \cite[Corollary 2.10]{FG1}, we have the following exact sequence $$0=Hom(X,G^{\ast})\to Hom(X,G^{\ast}/G)\to Ext(X,G)$$ Since $G^{\ast}/G$ is a discrete, torsion divisible group, so $Hom(X,G^{\ast}/G)$ containing a copy of $Hom(X,\Bbb Q/\Bbb Z)$. Hence by Lemma \ref{8}, $Ext(X,G)\neq 0$.
\end{proof}

\bibliographystyle{amsplain}

\end{document}